\documentclass[11pt, reqno]{amsart}%fleqn
\usepackage{amsmath}
\usepackage{amssymb}
\usepackage{amsfonts}
\usepackage{graphicx}
\usepackage{amsthm}
\usepackage{enumerate}
\usepackage{lscape}
\usepackage{dsfont}
\usepackage{color}
\usepackage{mathtools}
\usepackage[dvipsnames]{xcolor}
\usepackage{hyperref}

\usepackage[utf8]{inputenc}

\newcommand{\de}{\partial}          
\newcommand{\deb}{\overline\partial}

\newcommand{\ov}[1]{\overline{#1}}
\newcommand{\wi}[1]{\widetilde{#1}}
\newcommand{\hs}{\hspace{0.1em}}
\newcommand{\f}{\rightarrow}
\newcommand{\Aut}{\operatorname{Aut}}

\newcommand{\K}{K\"ahler}
\newcommand{\Ko}{Kobayashi}
\newcommand{\W}{\Omega}

\newcommand{\CH}{Cartan-Hartogs}

\newcommand{\Om}{\Omega}

\newcommand{\midd}{\ \big | \,}
\newcommand{\middd}{\ \Big | \,}
\newcommand{\midddd}{\ \bigg | \,}

\newtheorem{theorem}{Theorem}

\newtheorem{lemma}{Lemma}
\newtheorem{prop}{Proposition}

\newcommand{\C}{\mathds C}

\makeatletter
\@namedef{subjclassname@2020}{%
  \textup{2020} Mathematics Subject Classification}
\makeatother

\title{A Cartan--Hartogs version of the Polydisk Theorem}
\author{Roberto Mossa}
\address{(Roberto Mossa) Departamento de Matemática \\
 Instituto de Matemática e Estatística \\
Universidade de São Paulo (Brazil)}
	\email{robertom@ime.usp.br}

\author{Michela Zedda$^{(*)}$}
\address{(Michela Zedda - Corresponding author) Dipartimento di Matematica, Fisica e Informatica \\
          Universit\`a di Parma (Italy)}
\email{michela.zedda@unipr.it}

\date{\today}

\thanks{The first-named author was supported by a grant from Fapesp (2018/08971-9).
The second-named author has been financially supported by the group G.N.S.A.G.A. of I.N.d.A.M and by the PRIN project ``Real and Complex Manifolds: Topology, Geometry and holomorphic dynamics''.\\
\indent{Data sharing not applicable to this article as no datasets were generated or analyzed during the current study.}\\
\indent{(*) Corresponding author.}
}

\begin{document}
\maketitle
\begin{abstract}
We extend the Polydisk Theorem for symmetric bounded domains to Cartan--Hartogs domains, and apply it to prove that a Cartan--Hartogs domain inherits totally geodesic submanifolds from the bounded symmetric domain which is based on, and to give a characterization of Cartan--Hartogs's geodesics with linear support.
\end{abstract}

%\tableofcontents

\section{Introduction and statement of the results}

For a bounded symmetric domain $\Omega$ endowed with (a multiple of) its Bergman metric $g_B$, the celebrated Polydisk Theorem due to J. A. Wolf \cite{wolf} (see also \cite{mok}) shows that given any point $z\in \Omega$ and any direction $X\in T_z\Omega$, there exists a totally geodesic complex submanifold $\Pi$ passing through $z$ with $X\in T_z\Pi$, biholomorphically isometric to a polydisk $\Delta^r$  of dimension equals to the rank $r$ of $\Omega$. Moreover, the  group of the (isometric) automorphisms $\Aut(\Omega)$ of $\Omega$, acts transitively on the space of all such polydisks, and denoting by $\Aut_z(\Omega)$ the isotropy subgroup of $\Aut(\Omega)$ at $z$, one can realize $\Omega$ as union over $\gamma\in \Aut_z(\Omega)$ of $\gamma\cdot \Pi$.

In this paper, we prove the analogous of the Polydisk Theorem for Cartan-Hartogs domains in terms of Hartogs-Polydisk (see \eqref{hpoly} below). For $\mu>0$, Cartan--Hartogs domains are defined as the $1$-parameter family:
\begin{equation}\label{gench}
M_\Omega(\mu)=\left\{(z,w)\in\Omega\times \mathds C \midd |w|^2<N_\Omega^\mu(z,z)\right\},
\end{equation}
where $\Omega$ is a bounded symmetric domain not necessarily irreducible and $N_\Omega(z,z)$ is its generic norm. Observe that originally \cite{newclasses} the domain $\Omega$ the Cartan--Hartogs is based on is a Cartan domain, i.e. an irreducible bounded symmetric domain. Here $\Omega$ is allowed to be not irreducible, namely $\Om = \Om_1\times\cdots\times \Om_m$ is a product of the Cartan domains $\Om_1,\dots,\Om_m$, and accordingly its generic norm $N_\Omega$ is the product of the generic norms of each factor:
$$
N_\Om(z_1,\dots,z_m, z_1,\dots,  z_m)= N_{\Om_1}(z_1, z_1)\cdots N_{\Om_m}(z_m, z_m).
$$
We consider on $M_\Omega(\mu)$ the Kobayashi metric $\omega(\mu)=\frac i2\partial\bar \partial \Phi_{\Omega,\mu}$, where:
\begin{equation}\label{genphi}
\Phi_{\Omega,\mu}(z,w)=-\log\left(N_\Omega^\mu(z,z)-|w|^2\right).
\end{equation}
We say that a Cartan--Hartogs $M_{\Om}(\mu)$ domain is of \emph{classical type} if $\Om$ is product of Cartan domains of classical type.
When $\Omega$ is a polydisk $\Delta^n:=\{z\in \mathds C^n \mid |z_1|^2<1,\dots, |z_n|^2<1\}$, the associated Cartan--Hartogs is the {\em Hartogs-Polydisk}:
\begin{equation}\label{hpoly}
M_{\Delta^n}(\mu)=\left\{(z,w)\in \Delta^n\times \mathds C \middd |w|^2<\prod_{j=1}^n(1-|z_j|^2)^\mu\right\},
\end{equation}
whose Kobayashi metric is defined by the K\"ahler potential:
$$
\Phi_{\Delta^n,\mu}(z,w)=-\log\left(\prod_{j=1}^n(1-|z_j|^2)^\mu-|w|^2\right).
$$
Observe that when  $\mu=1$ and $\Omega=\mathds C{\rm H}^1$, $M_\Omega(\mu)$ reduces to be a complex hyperbolic space. In all the other cases it is a nonhomogeneous domain
 that inherits symmetric peculiarities from the symmetric bounded domain it based on. For this reason Cartan--Hartogs domains represent an important class of domains in $\mathds C^n$, and since their first apparence in \cite{newclasses} they have been studied from different points of view, see e.g. \cite{bi,feng,fengtu, LZ, LZzeith,MZ, roosKE, zedda, zeddaberezin}. 
 
The main theorem of this paper is the following Hartogs version of the Polydisk Theorem. As his classical counterpart, which led to several applications, e.g. N. Mok and S.-C. Ng's rigidity and extension results for holomorphic isometries \cite{mokgeod, mokarithm, mokholisom,mokgerm,ng} (see also \cite{LM,M1,M2} where the Polydisk Theorem is used to study the diastatic exponential and the volume and diastatic entropy of symmetric bounded domains), we expect it to be a useful tool to solve geometric problems related to Cartan--Hartogs domain, improving our knowledge of nonhomogeneous domains.
\begin{theorem}[Hartogs--Polydisk Theorem]\label{thmpoly}
Let $\Omega$ be a bounded symmetric domain of classical type of rank $r$ and let $M_\Omega(\mu)$ be the associated Cartan--Hartogs domain.
%$M_{\Omega}(\mu)$ be  a Cartan-Hartogs domain of classical type, with $\rank\left(\Omega\right)=r$. 
For any point $(z,w) \in M_ \Omega(\mu)$ and any $X \in T_{(z,w)}M_ \Omega(\mu)$ there exists a totally geodesic complex submanifold $\tilde \Pi$ through $(z,w)$ with $X \in T_{(z,w)}M_ {\tilde \Pi}(\mu)$, such that $\tilde \Pi$ is biholomorphically isometric to the Hartogs-Polydisk $M_{\Delta^r}(\mu)$.
Moreover, $\Aut(\Omega)$ acts transitively on the space of all such Hartogs-polydisks, and
$
M_{\Omega}(\mu)  = \cup \{\gamma\cdot \tilde \Pi: \gamma \in \Aut_z(\Omega)\}.
$
\end{theorem}

We apply the Hartogs-Polydisk Theorem to prove the following two results. The first one states that any totally geodesic K\"ahler submanifold of the base domain $\Omega$ is a totally geodesic submanifold of its associated Cartan--Hartogs: 
\begin{theorem}\label{thmtotg} 
Let $\Omega'\subset\Omega$ be a totally geodesic \K\ submanifold of a bounded symmetric domain of classical type. Then 
$$
C_{\W'}=\left\{\left(z,w\right)\in M_{\Omega}(\mu) \mid z\in \Omega'\right\}
$$
is a totally geodesic K\"ahler submanifold of $M_{\Omega}(\mu)$ biholomorphically isometric to the \CH\ $M_{\Omega'}(\mu)$.
\end{theorem}
The second one gives a characterization of geodesics with linear support in $M_{\Omega}(\mu)$:

\begin{theorem}\label{autlingeo}
If $M_\Omega(\mu)$ admits a geodesic with linear support passing through $(\zeta,0)$, then up to automorphisms either the geodesic is contained in $\Omega= M_\Omega(\mu)\cap \{w=0\}$ or in $\mathds C{\rm H}^1=M_\Omega(\mu)\cap \{z=0\}$, or $M_\Omega(\mu)\simeq\mathds C{\rm H}^{d+1}$.
\end{theorem}

The paper is organized as follows. In the next section we recall basic facts about classical Cartan domains and we describe explicit polydisks totally geodesically embedded. In Section \ref{proof} we show how the totally geodesic K\"ahler immersions of such polydisks into the Cartan domains lift to totally geodesic K\"ahler immersion of Hartogs--polydisks into Cartan--Hartogs domains and prove Theorem \ref{thmpoly}. The last three sections are devoted respectively to the proofs of theorems \ref{thmtotg} and \ref{autlingeo}.

\section{Explicit polydisks in Cartan domains}\label{cartan}

In this section we are going to give an explicit totally geodesic K\"ahler (i.e. holomorphic and isometric) immersion of a polydisk into each one of the four irreducible classical domains. All the isometries here are intended respect to the hyperbolic metric on $\Omega$, i.e. $\omega_{hyp}^\Omega:=-\partial\bar\partial\log N_\Omega(z,z)$ (one has $\omega^\Omega_B=\gamma\omega^\Omega_{hyp}$, where $\omega_B^\Omega$ is the Bergman metric on $\Omega$ and $\gamma$ is its genus). Throughout this section we use the Jordan triple system theory, referring the reader  to \cite{loiscala,LM,lmz,bisy,loos,M1,M2,M3,MZ,roos} for details and further applications.

\subsection{Cartan domain of the first type}\label{firstdomain}
Consider the first Cartan domain of rank $r=m$ and genus $\gamma=n+m$:
$$
\Omega_1[m,n] = \left\{ Z \in M_{m,n}(\C) \mid \det \left(I_m - Z Z^* \right) > 0,\, n\geq m  \right\}.
$$
Its generic norm is given by:
\begin{equation}\label{genericnorm1}
N_{\Omega_{1}}(Z, Z)={\det}\left(I_{m}-Z Z^{*}\right).
\end{equation}
A totally geodesic polydisk $\Delta^m \xhookrightarrow{\varphi} \Omega_1[m,n]$ is given by
\begin{equation}\label{phizeta1}
{\varphi}(z_1,\dots,z_m) = \operatorname {diag}(z_1, \dots, z_m){=\begin{pmatrix}z_1&&&0\\ &\ddots&&\\ &&z_m&0\end{pmatrix}}.
\end{equation}
Since $\det(I_m-\varphi(z)\varphi(z)^*)=\prod_{j=1}^m(1-|z_j|^2)$, $\varphi$ is clearly a \K\ immersion. Moreover it is easy to check that $\varphi_*\left(T_0\Delta^m\right)$ define a sub-HJPTS of $\left(T_0\Omega_1[{m,n}], \left\{,,\right\}\right)$, where 
\begin{equation}\label{eqtriprod}\begin{split} 
 \left\{U,V,W\right\}=UV^*W+WV^*U
\end{split}\end{equation}
 (see e.g. \cite[(16)]{loiscala}), we conclude, by the one to one correspondence between sub-HJPTS e sub-HSSNT (see \cite[Proposition 2.1]{loiscala}), that $\varphi$ is totally geodesic. Recent application

 \subsection{Cartan domain of the second type}\label{seconddomain}
Consider the second Cartan domain of rank $r=[n/2]$ and genus {$\gamma=2n+2$},
$$
\Omega_{2}[n]=\left\{Z \in M_{n}(\mathbb{C}), Z=-Z^{T}, \det (I_{n}-Z Z^{*})>0\right\}.
$$
A parametrization is given by:
$$
u= (u_{1\,2},\dots,u_{1\,n},u_{2\,3},\dots,u_{2\,n}\dots,u_{n-1\, n} ) \mapsto Z(u) = \left(\begin{smallmatrix}
 0            &  u_{1\,2}  & u_{1\,3} & \dots  &  u_{1\, n-1} & u_{1\,n} \\
 -u_{1\,2}  &  0           & u_{2\,3} & \dots  &  u_{2\, n-1} & u_{2\,n} \\
 \vdots    & \vdots    & \vdots  & \vdots & \vdots           &  \vdots  \\
 -u_{1\,n}  & -u_{2\,n}  &-u_{3\,n} & \dots  & -u_{n-1\, n} & 0 .     
\end{smallmatrix}\right)
$$
Its generic norm is given by:
\begin{equation}\label{genericnorm2}
N_{\Omega_{2}}(u, u)={\det}^{1/2}\left(I_{n}-Z(u) Z^{*}(u)\right).
\end{equation}

A totally geodesic polydisk $\Delta^{\left[\frac n2\right]} \xhookrightarrow{\varphi} \Omega_2[n]$ is given by:
\begin{equation}\label{phi2}
{\varphi}(u) = \begin{pmatrix}
 0            &  0   & \dots  &  0 & u_{1\,  \left[\frac n2\right]} \\
 0  &  0            & \dots  &  u_{2\, \left[\frac n2\right]-1} & 0 \\
 \vdots    & \vdots      & \vdots & \vdots           &  \vdots  \\
 0  & -u_{2\, \left[\frac n2\right]-1}  & \cdots & 0 & 0  \\          
 -u_{1\,  \left[\frac n2\right]}  & 0   & \cdots & 0 & 0             
\end{pmatrix},
\end{equation}
where $u=\left(u_{1\,  \left[\frac n2\right]},u_{2\,  \left[\frac n2\right]-1},\dots, u_{ \left[\frac n2\right]\, 1}\right)$. Since
$$
N_{\Omega_{2}}(\varphi(u), \varphi(u))={\det}^{1/2}\left(I_{n}-\varphi(u) \varphi^{*}(u)\right)=\prod_{j=1}^{ \left[\frac n2\right]}(1-|u_{j\, [\frac n2-j+1]}|^2),
$$
 $\varphi$ is a \K\ immersion, moreover it is easy to check that $\varphi_*\left(T_0\Delta_m\right)$ defines a sub-HJPTS of $\left(T_0\Omega_{2}[n], \left\{,,\right\}\right)$, where the triple product is given by $\left\{U,V,W\right\}=UV^*W+WV^*U$, namely the restriction to $T_0\Omega_{2}[n]$ of the triple product of $T_0\Omega_1[n,n]$ given in \eqref{eqtriprod}, we conclude, by the one to one correspondence between sub-HJPTS e sub-HSSNT (see \cite[Proposition 2.1]{loiscala}), that $\varphi$ is totally geodesic.

  %%%%%%%%%%% %%%%%%%%%%% %%%%%%%%%%% %%%%%%%%%%%
\subsection{Cartan domain of the third type}\label{thirddomain}
 Consider the Cartan domain of third type of rank $r=m$ and genus $\gamma=n+1$:
 \begin{equation*}
\Omega_{3}[m]=\left\{Z \in M_{m}(\mathbb{C}) \mid Z=Z^{T},\, \det (I_{m}-Z Z^{*})>0\right\},
\end{equation*}
whose generic norm is given by: 
\begin{equation}\label{genericnorm3}
N_{\Omega_{3}}(z, z)={\det}\left(I_{m}-Z Z^{*}\right).
\end{equation}
As can be proven in a totally similar way as done for the first and second type domains, a totally geodesic polydisk $\Delta^{m} \xhookrightarrow{\varphi} \Omega_1[m]$ is given by:
\begin{equation}\label{phizeta3}
{\varphi}(z) = \operatorname {diag}(z_1, \dots, z_m){=\begin{pmatrix}z_1&&\\ &\ddots&\\ &&z_m\end{pmatrix}}.
\end{equation}

\subsection{Cartan domain of the fourth type}\label{fourthdomain}
Consider the fourth type domain of rank $r=2$ and genus $\gamma=n$:
$$
\Omega_{4}[n]=\left\{z=\left(z_{1}, \dots, z_{n}\right) \in \mathbb{C}^{n} \midddd  \sum_{j=1}^{n}\left|z_{j}\right|^{2}<1,\, 1+\left|\sum_{j=1}^{n} z_{j}^{2}\right|^{2}-2 \sum_{j=1}^{n}\left|z_{j}\right|^{2}>0, \, n \geq 5\right\}.
$$
whose generic norm is given by:
\begin{equation}\label{genericnorm4}
N_{\Omega_{4}[n]}(z, z)=1+\left|\sum_{j=1}^{n} z_{j}^{2}\right|^{2}-2 \sum_{j=1}^{n}\left|z_{j}\right|^{2}.
\end{equation}
 Let $\varphi:\Delta^2\rightarrow \Omega_4[n]$ be the map: 
\begin{equation}\label{phizeta4}
\varphi(z_1,z_2)=\left(\frac{1}{2}\left(z_1+z_2\right),\frac{i}{2}\left(z_1-z_2\right) ,0,\dots,0 \right).
\end{equation}
Since:
$$
N_{\Omega_{4}[n]}(\varphi(z_1,z_2),\varphi(z_1,z_2))=(1-|z_1|^2)(1-|z_2|^2)=N_{\Delta^2}(z_1,z_2),
$$
$\varphi$ is K\"ahler. Moreover $\varphi\left(\Delta^2\right)$ is the set of points of $\Omega_{4}[n]$ fixed by the isometry $\left(z_1,\dots,z_n\right) \mapsto \left(z_1,z_2,-z_3\dots,-z_n\right)$, thus $\varphi$ is totally geodesic.

 \section{The Polydisk Theorem for Cartan-Hartogs domains}\label{proof}
Let us begin with the following lemma.
\begin{lemma}\label{ki}
Let  $\Omega$ be a Cartan domain and let $\varphi\!:\Delta^r\rightarrow \Omega$ be a K\"ahler immersion fixing the origin, i.e. a holomorphic map satisfying $\varphi^*\omega_{hyp}^\Omega=\omega_{hyp}^{\Delta^r}$ and $\varphi(0)=0$. Then:
$$
f\!:M_{\Delta^r}(\mu)\rightarrow M_\Omega(\mu), \quad f(z,w)=(\varphi(z),w),
$$
is a K\"ahler immersion.
\end{lemma}
\begin{proof}
Observe that $-\log(V(\Omega) N_\Omega(z,z))$ and $-\log(V(\Delta^r) N_{\Delta^r}(z,z))$ are the diastasis functions respectively for $(\Omega,\omega_{hyp}^\Omega)$ and $(\Delta^r,\omega_{hyp}^{\Delta^r})$, where $V(\Omega)$ (resp. $V(\Delta^r)$) is the total volume of $\Omega$ (resp. $\Delta^r$) with respect to the Euclidean measure of the ambient complex Euclidean space (see \cite[Prop. 7]{LZ}). Since the diastasis is a K\"ahler potential invariant by isometries (see \cite{Cal} or also \cite{LZbook}), one has:
$$
V(\Omega) N_\Omega(\varphi(z),\varphi(z))=V(\Delta^r)N_{\Delta^r}(z,z).
$$
Since $\varphi(0)=0$ and $N_\Omega(0,0)=1$ for any $\Omega$, we get $V(\Omega) =V(\Delta^r)$ and thus
 $N_\Omega(\varphi(z),\varphi(z))=N_{\Delta^r}(z,z)$. Then, it follows easily that:
$$
\Phi_{\Omega,\mu}(f(z,w))=-\log\left(N_\Omega((\varphi(z),\varphi(z))-|w|^2\right)=-\log\left(N_{\Delta^r}(z,z)-|w|^2\right)=\Phi_{\Delta^r,\mu}(z,w),
$$
and we are done.
\end{proof}

By this lemma the totally geodesic K\"ahler immersions described in the previous section induce K\"ahler immersions of Hartogs--polydisks into Cartan--Hartogs domains. We prove now case by case that such maps are also totally geodesics.

\subsection{Cartan--Hartogs domain of the first type}
By \eqref{gench}, \eqref{genphi} and \eqref{genericnorm1}, the Cartan-Hartogs domain associated to a first type Cartan domain is:
$$
M_{\Omega_1[m,n]}(\mu)=\left\{(z,w)\in \Omega_1[m,n]\times \mathds C \mid  |w|^2<{\det}^\mu\left(I_{m}-Z Z^{*}\right)\right\}.
$$
and its Kobayashi metric is described by the K\"ahler potential:
$$
\Phi_{\Omega_1,\mu}(z)=-\log\left({\det}^\mu\left(I_{m}-Z Z^{*}\right)-|w|^2\right).
$$
\begin{lemma}\label{lempolytot}
Let $\varphi\!:\Delta^m\rightarrow \Omega_1[m,n]$ be the map in \eqref{phizeta1}. Then $f: M_{\Delta^m}(\mu) \rightarrow M_{\Omega_1[m,n]}(\mu)$, $f(z,w) = (\varphi(z) , w)$,
is a totally geodesic K\"ahler immersion.
\end{lemma}
\begin{proof}
From Sec. \ref{firstdomain} the map $\varphi$ is a K\"ahler immersion, thus by Lemma \ref{ki} also $f$ is.

 It remains to prove that $f$ is totally geodesic. Let $Z=(z_{jk})$. From the expression of $f$ and \eqref{phizeta1}, $\{\de_{z_{jj}}, \hs \de_{z_{kk}}, \hs \de_{w}\}$ is a basis for $T f\! \left(M_{\Delta^m}(\mu)\right) \subset T M_{\Omega_1[m,n]}(\mu)$. Thus,  it is enough to show that:
\begin{equation}\label{eqcov1}
\nabla _{\de_{z_{jj}}}{\de_{z_{kk}}}, \nabla _{\de_{w}}{\de_{z_{kk}}}  \in T f\! \left(M_{\Delta^m}(\mu)\right) \quad j,k= 1, \dots, n .
\end{equation}
Recalling that the covariant derivative in terms of Christoffel symbols reads:
$$
\nabla _{\de_{z_{jj}}}{\de_{z_{kk}}} = \sum_{s=1}^m  \sum_{r=1}^n\left( \Gamma_{jj\hs kk}^{sr} \de_{z_{sr}}+\Gamma_{jj\hs kk}^{0} \de_{w}\right) \quad \text{and} \quad \nabla _{\de_w}{\de_{z_{kk}}} =  \sum_{s=1}^m  \sum_{r=1}^n \left(\Gamma_{0\hs kk}^{sr} \de_{z_{sr}}+\Gamma_{0\hs kk}^{0} \de_{w}\right),
$$
where we use the index $0$ for the $w$-entry, \eqref{eqcov1} is equivalent to $\Gamma_{jj\hs kk}^{sr} =\Gamma_{0\hs kk}^{sr}= 0 $ for $s\neq r$. In order to compute the Christoffel symbols let $A=I-ZZ^*$ and denote by $\wi A_{r_{11}r_{12},\dots,r_{s1}r_{s2}}$ the matrix obtained from $A$ by removing the $r_{11},\dots,r_{s1}$-th rows and the $r_{12},\dots,r_{s2}$-th columns. If $j\neq r_{11},\dots,r_{s1}$ we have
$$
\frac{\de \det \wi A_{r_{11}r_{12},\dots,r_{s1}r_{s2}} }{\de z_{jk}}= \sum_\ell \epsilon_{j\ell}\det \wi A_{r_{11}r_{12},\dots,r_{s1}r_{s2},j\ell}\hs\ov z_{\ell k},
$$
where $\epsilon_{j\ell}$ is an opportune constant equal to $1$ or $-1$. Clearly if the hypothesis $j\neq r_{11},\dots,r_{s1}$ is not satisfied the derivative is zero. Similarly if $j\neq r_{12},\dots,r_{s2}$ we have
$$
\frac{\de \det \wi A_{r_{11}r_{12},\dots,r_{s1}r_{s2}} }{\de \ov z_{jk}}= \sum_\ell \epsilon_{\ell j}\det \wi A_{r_{11}r_{12},\dots,r_{s1}r_{s2},\ell j}\hs z_{\ell k},
$$
zero otherwise. Thus we have
$$
\frac{\de\Phi_{\Omega_1,\mu}}{\de \ov z_{rs}}
=- \frac{\de\log(\det^\mu\!\! A -|w|^2)}{\de \ov z_{rs}}
=\frac{\mu\det^{\mu-1}\!\! A  \sum_{\ell} \epsilon_{\ell r} \det \wi{A}_{\ell r}\hs z _{\ell s}}{\det^\mu \!\! A  - |w|^2},
$$
\begin{equation*}\begin{split} 
 \frac{\de^2\Phi_{\Omega_1,\mu}}{\de w \hs \de \ov z_{rs}}=   \frac{\ov w \mu \det^{\mu-1}\!\! A  \sum_{\ell} \epsilon_{\ell r} \det \wi{A}_{\ell r}\hs z _{\ell s}}{\left(\det^\mu \!\!  A  - |w|^2\right)^ 2} 
 \end{split}\end{equation*}
and
\begin{equation}\label{eqgjkrs}\begin{split} 
 \frac{\de^2\Phi_{\Omega_1,\mu}}{\de z_{jk} \hs \de \ov z_{rs}}&=   \frac{\mu(\mu-1)\det^{\mu-2}\!\! A  \sum_{\ell} \epsilon_{j\ell} \det \wi{A}_{j\ell}\hs \ov z _{\ell k}\sum_{\ell} \epsilon_{\ell r} \det \wi{A}_{\ell r}\hs z _{\ell s}}{\det^\mu \!\!  A  - |w|^2} \\
& +  \frac{\mu\det^{\mu-1}\!\! A  \left(\sum_{\ell,t} \epsilon_{\ell r}  \epsilon'_{jt}\det \wi A_{\ell r,j t}\hs\ov z_{t k}\hs z _{\ell s} +\epsilon_{j r} \det \wi{A}_{j r}\hs \delta_{ks}\right)}
{\det^\mu \!\! A  - |w|^2}\\
& - \frac{\left(\mu\det^{\mu-1}\!\! A  \right)^2 \sum_{\ell} \epsilon_{\ell r} \det \wi{A}_{\ell r}\hs z _{\ell s} \sum_{\ell} \epsilon_{j\ell} \det \wi{A}_{j\ell}\hs \ov z _{\ell k}}{\left(\det^\mu \!\! A  - |w|^2\right)^2}.
 \end{split}\end{equation}
In particular
 \begin{equation*}\begin{split} 
\left( \frac{\de^2\Phi_{\Omega_1,\mu}}{\de w \hs \de \ov z_{rs}}\right)_{\mid Z= \operatorname {diag}(z_{11}, \dots, z_{mm})}= 0 \qquad \text{for } r\neq s. 
 \end{split}\end{equation*}
and  
\begin{equation*}\begin{split} 
\left( \frac{\de^2\Phi_{\Omega_1,\mu}}{\de z_{jk} \hs \de \ov z_{rs}} \right)_{\mid Z= \operatorname {diag}(z_{11}, \dots, z_{mm})}= 0 \qquad \text{for } r\neq s \text{ and } (j,k)\neq (r,s). 
 \end{split}\end{equation*}
Therefore, for  $(Z,w) \in  f\! \left(M_{\Delta^m}(\mu)\right)$, and $s \neq r$, we have
$$
\Gamma_{0\hs kk}^{rs} = \sum_\ell g^{rs,\ov \ell}\hs \frac{\de g_{kk,\ov\ell}}{\de w}=g^{rs,\ov {rs}}\hs \frac{\de g_{kk,\ov {rs}}}{\de w}=0,
$$
where the last equality is easily deduced from \eqref{eqgjkrs}. 
It remains to prove that for $(Z,w) \in  f\! \left(M_{\Delta^m}(\mu)\right)$, and $s \neq r$, we have
$$
\Gamma_{jj\hs kk}^{rs} = \sum_\ell g^{rs,\ov \ell}\hs \frac{\de g_{kk,\ov\ell}}{\de z_ {jj}}=g^{rs,\ov {rs}}\hs \frac{\de g_{kk,\ov {rs}}}{\de z_ {jj}}=0,
$$
i.e. $ \frac{\de^3\Phi_{\Omega_1,\mu}}{\de z_{jj} \hs \de z_{kk} \hs \de \ov z_{sr}}=0$. Assume $s\neq r$, we have
{\small \begin{equation*}\begin{split} 
 \left( \frac{\de^3\Phi_{\Omega_1,\mu}}{\de z_{jj} \hs \de z_{kk} \hs \de \ov z_{sr}} \right)_{\mid Z= \operatorname {diag}(z_{11}, \dots, z_{mm})}
=& \frac{\de}{\de z_{jj} }
 \left(   \frac{\mu(\mu-1)\det^{\mu-2}\!\! A \,  \epsilon_{kk} \det \wi{A}_{kk}\hs \ov z _{k k} \,\epsilon_{s r} \det \wi{A}_{s r}\hs z _{s s}}{\det^\mu \!\!  A  - |w|^2} \right.\\
& +  \frac{\mu\det^{\mu-1}\!\! A \left(\epsilon_{s r}  \epsilon_{kk}\det \wi A_{s r,k k}\hs\ov z_{k k}\hs z _{s s} +  \epsilon_{k r} \det \wi{A}_{k r}\hs  \delta_{ks}\right)}
{\det^\mu \!\! A  - |w|^2}\\
&\left. - \frac{\left(\mu\det^{\mu-1}\!\! A  \right)^2  \epsilon_{s r} \det \wi{A}_{s r}\hs z _{s s} \epsilon_{kk} \det \wi{A}_{kk}\hs \ov z _{k k}}
{\left(\det^\mu \!\! A  - |w|^2\right)^2}\right)_{\mid Z= \operatorname {diag}(z_{11}, \dots, z_{mm})}\\
% \end{split}\end{equation*}
% \begin{equation*}\begin{split} 
=& 
 \left(   \frac{\mu(\mu-1)\det^{\mu-2}\!\! A \,  \epsilon_{kk} \det \wi{A}_{kk}\hs \ov z _{k k} \,\epsilon_{s r} \frac{\de \det \wi{A}_{s r}}{\de z_{jj} } \hs z _{s s}}{\det^\mu \!\!  A  - |w|^2} \right.\\
& +  \frac{\mu\det^{\mu-1}\!\! A  \left(\epsilon_{s r}  \epsilon_{kk}\frac{\de \det \wi A_{s r,k k}}{\de z_{jj} }\hs\ov z_{k k}\hs z _{s s} +   \epsilon_{k r} \frac{\de \det \wi{A}_{k r}}{\de z_{jj} } \hs  \delta_{ks}\right)}
{\det^\mu \!\! A  - |w|^2}\\
&\left. - \frac{\left(\mu\det^{\mu-1}\!\! A  \right)^2  \epsilon_{s r} \frac{\de \det \wi{A}_{s r}}{\de z_{jj} }\hs z _{s s} \epsilon_{kk} \det \wi{A}_{kk}\hs \ov z _{k k}}
{\left(\det^\mu \!\! A  - |w|^2\right)^2}\right)_{\mid Z= \operatorname {diag}(z_{11}, \dots, z_{mm})}\\
=&0,
 \end{split}\end{equation*}}
concluding the proof.
\end{proof}

\subsection{Cartan--Hartogs domain of the second type}
By \eqref{gench}, \eqref{genphi} and \eqref{genericnorm2}, the Cartan--Hartogs associated to the second type Cartan domain is:
$$
M_{\Omega_2[n]}(\mu)=\{(u,w)\in \Omega_2[n]\times \mathds C \mid  |w|^2<\operatorname{det}^{\mu/2}\!\left(I_{m}-Z(u) Z(u)^{*}\right)\},
$$
and a K\"ahler potential for its Kobayashi metric is:
$$
\Phi_{\Omega_2,\mu} (u,w)= -\log\left( \operatorname{det}^{\mu/2}\!\left(I_{n}-Z(u) Z(u)^{*}\right)-|w|^2\right).
$$ 
%Notice that $\phi_{\Omega_2[n]} (u,w) = \phi_{\Omega_1[n]} (Z(u),w)$ once substituted  $\mu$ with $\mu/2$ in the second term.

\begin{lemma}\label{lempolytot2}
Let $\varphi\!:\Delta^{\frac n2}\rightarrow \Omega_2[n]$ be the map in \eqref{phi2}. Then $f: M_{\Delta^{n/2}}(\mu) \rightarrow M_{\Omega_2[n]}(\mu)$, $f(u, w)= \left(\varphi(u),w\right)$,
is a totally geodesic  K\"ahler immersion.
\end{lemma}
\begin{proof}
From Sec. \ref{seconddomain} the map $\varphi$ is K\"ahler, thus by Lemma \ref{ki} $f$ also is.
%We proceed as in the proof of Lemma \ref{lempolytot}. Since $\varphi$ is K\"ahler and in particular  $N_{\Delta^{[n/2]}}(u,u)=N_{\ii}(\varphi(u),\varphi(u))$, it follows that $f$ is a well defined holomorphic isometric immersion. 
Let us use the parametrization described in Sec. \ref{seconddomain}. In terms of the Christoffel symbols, since:
$$
\nabla _{\de_{u_{j\, n+1-j}}}{\de_{u_{k\, n+1-k}}} = \sum_{s,r=1}^n \Gamma_{j\, n+1-j,\, k\, n+1-k}^{sr} \de_{u_{sr}}+\Gamma_{j\,n+1-j,\, k\, n+1-k}^{0} \de_{w}, $$
and 
$$
 \quad \nabla _{\de_w}{\de_{u_{ k \, n+1-k}}} = \sum_{s,r=1}^n \Gamma_{0,\, k \,n+1-k}^{sr} \de_{u_{sr}}+\Gamma_{0,\, k \, n+1-k}^{0} \de_{w},
$$
the map $f$ to be totally geodesic is equivalent to:
\begin{equation}\label{eqgam4}
\Gamma_{j\, n+1-j,\, k\,  n+1-k}^{sr}=\Gamma_{0,\, k \, n+1-k}^{sr}= 0 ,
\end{equation}
 for $s\neq n+1 - r$ and $1\leq j,k \leq n/2$. Observing that $\Phi_{\Omega_2,\mu} (u,w) = \Phi_{\Omega_1,\mu} (Z(u),w)$ once substituted  $\mu$ with $\mu/2$ in the second term,
we have:
\begin{equation*}\begin{split} 
\frac{\de\Phi_{\Omega_2,\mu}}{\de \ov u_{rs}} =& \frac{\de\Phi_{\Omega_1,\mu}}{\de \ov z_{rs}} - \frac{\de\Phi_{\Omega_1,\mu}}{\de \ov z_{sr}}
=- \frac{\de\log(\det^{\frac{\mu}{2}}\!\! A -|w|^2)}{\de \ov z_{rs}}+\frac{\de\log(\det^{\frac{\mu}{2}}\!\! A -|w|^2)}{\de \ov z_{sr}}\\
=&\frac{{\frac{\mu}{2}}\det^{{\frac{\mu}{2}}-1}\!\! A  \left(\sum_{\ell=1}^n \epsilon_{\ell r} \det \wi{A}_{\ell r}\hs z _{\ell s}(u)-\sum_{\ell=1}^n \epsilon_{\ell s} \det \wi{A}_{\ell s}\hs z _{\ell r}(u)\right)}{\det^{\frac{\mu}{2}} \!\! A  - |w|^2},
 \end{split}\end{equation*}
and
 \begin{equation}\label{eqwu2}\begin{split} 
\frac{\de^2\Phi_{\Omega_2,\mu}}{\de w \hs \de \ov u_{rs}} &= \frac{\de^2\Phi_{\Omega_1,\mu}}{\de w \hs \de \ov z_{rs}}-\frac{\de^2\Phi_{\Omega_1,\mu}}{\de w \hs \de \ov z_{sr}}\\
&=   \frac{\ov w {\frac{\mu}{2}} \det^{{\frac{\mu}{2}}-1}\!\! A  \left(\sum_{\ell=1}^n \epsilon_{\ell r} \det \wi{A}_{\ell r}\hs z _{\ell s}(u)-\sum_{\ell=1}^n \epsilon_{\ell s} \det \wi{A}_{\ell s}\hs z _{\ell r}(u)\right)}{\left(\det^{\frac{\mu}{2}} \!\!  A  - |w|^2\right)^ 2}.
 \end{split}\end{equation} 
 For $(u,w)\in  f (\Delta_{[n/2]})$ and $r\neq n+1-s$, we get
 $$
 \frac{\de^2\phi_2}{\de w \hs \de \ov u_{rs}} = \frac{{\frac{\mu}{2}}\det^{{\frac{\mu}{2}}-1}\!\! A  \left( \epsilon_{n+1-r\, r} \det \wi{A}_{n+1-s\, r}\hs z _{n+1-s\, s}(u)-\epsilon_{n+1-s\, s} \det \wi{A}_{n+1-r\, s}\hs z _{n+1-r\, r}(u)\right)}{\det^{\frac{\mu}{2}} \!\! A  - |w|^2}=0.
 $$
Notice that
 \begin{equation}\label{thirdterm}\begin{split} 
 \frac{\de^2\Phi_{\Omega_2,\mu}}{\de u_{jk} \hs \de \ov u_{rs}} &=\left(\frac{\de\Phi_{\Omega_1,\mu}}{\de  z_{jk}} - \frac{\de\Phi_{\Omega_1,\mu}}{\de  z_{kj}}\right) \left(\frac{\de\Phi_{\Omega_1,\mu}}{\de  \ov z_{rs}} - \frac{\de\Phi_{\Omega_1,\mu}}{\de \ov z_{sr}}\right)\\
 &= \frac{\de^2\Phi_{\Omega_1,\mu}}{\de z_{jk} \hs \de \ov z_{rs}} + \frac{\de^2\Phi_{\Omega_1,\mu}}{\de z_{kj} \hs \de \ov z_{sr}} -\frac{\de^2\Phi_{\Omega_1,\mu}}{\de z_{kj} \hs \de \ov z_{rs}} - \frac{\de^2\Phi_{\Omega_1,\mu}}{\de z_{jk} \hs \de \ov z_{sr}}.
 \end{split}\end{equation}

 If we take $\ov u_{rs}$ with $r\neq n+1-s$ and $u_{jk}$ with $(j,k)\neq(r,s)$, then the indexes of $\frac{\de^2\Phi_{\Omega_1,\mu}}{\de z_{jk} \hs \de \ov z_{rs}}$, $ \frac{\de^2\Phi_{\Omega_1,\mu}}{\de z_{kj} \hs \de \ov z_{sr}}$, $\frac{\de^2\Phi_{\Omega_1,\mu}}{\de z_{kj} \hs \de \ov z_{rs}}$, $ \frac{\de^2\Phi_{\Omega_1,\mu}}{\de z_{jk} \hs \de \ov z_{sr}} $ in \eqref{thirdterm} must satisfies $r\neq n+1-s$, $(j,k)\neq(r,s)$ and $(k,j)\neq(r,s)$. Under this conditions on the indexes, it is just a straightforward computation to prove that 
  $
  \left( \frac{\de^2\Phi_{\Omega_1,\mu}}{\de z_{jk} \hs \de \ov z_{rs}}\right)_{\mid (u,w)\in  f (\Delta_{[n/2]})}
  =\left( \frac{\de^2\Phi_{\Omega_1,\mu}}{\de z_{kj} \hs \de \ov z_{sr}}\right)_{\mid (u,w)\in  f (\Delta_{[n/2]})}
  =\left( \frac{\de^2\Phi_{\Omega_1,\mu}}{\de z_{kj} \hs \de \ov z_{rs}}\right)_{\mid (u,w)\in  f (\Delta_{[n/2]})}
  =\left( \frac{\de^2\Phi_{\Omega_1,\mu}}{\de z_{jk} \hs \de \ov z_{sr}}\right)_{\mid (u,w)\in  f (\Delta_{[n/2]})}=0,
  $
 and in particular that $\left( \frac{\de^2\Phi_{\Omega_2,\mu}}{\de u_{jk} \hs \de \ov u_{rs}}\right)_{\mid Z(u)\in \Delta}=0$. We conclude that for $(u,w)\in  f (\Delta_{[n/2]})$ and $r\neq n+1-s$ we have
$$
\Gamma_{0,\hs k \hs n+1-k}^{sr}= \sum_\ell g^{rs,\ov \ell}\hs \frac{\de g_{k \hs n+1-k,\hs\ov\ell}}{\de w}=g^{rs,\ov {rs}}\hs \frac{\de g_{k \hs n+1-k,\,\ov {rs}}}{\de w},
$$
 and 
 $$
\Gamma_{n+1-j\hs j, n+1-k\hs k}^{sr} = \sum_\ell g^{rs,\ov \ell}\hs \frac{\de g_{ n+1-k\hs k,\ov\ell}}{\de u_ {n+1-j\hs j}}=g^{rs,\ov {rs}}\hs \frac{\de g_{ n+1-k\hs k,\ov {rs}}}{\de u_ {n+1-j\hs j}}.% \quad \text{for } Z(u)\in \Delta \text{ and } r\neq n+1. 
$$
Deriving \eqref{eqwu2}, we can see that $\frac{\de g_{k \hs [n+1-k],\,\ov {rs}}}{\de w}=0$, which readily implies that $\Gamma_{0,\hs k \hs n+1-k}^{sr}=0$.
It remains to prove that, under the above conditions $\frac{\de g_{n+1-k\hs k,\ov {rs}}}{\de u_ {n+1-j\hs j}}=0$ (or equivalently that $\frac{\de^3\phi_2}{\de u_ {n+1-j\hs j}\hs \de u_{n+1-k\hs k} \hs \de \ov u_{rs}}=0$). We have

  \begin{equation}\label{eqgjkrs}\begin{split} 
 &\frac{\de^3\Phi_{\Omega_1,\mu}}{\de z_{j\hs [n+1-j]}\hs \de z_{k\hs[ n+1-k]} \hs \de \ov z_{rs}}=  \frac{\de}{\de z_{j\hs[ n+1-j]}}\left(\frac{\de^2\Phi_{\Omega_1,\mu}}{\de z_{k\hs[ n+1-k]} \hs \de \ov z_{rs}}\right)\\
 &= \frac{\de}{\de z_{j\hs [n+1-j]}}\left( \frac{{\frac{\mu}{2}}({\frac{\mu}{2}}-1)\det^{{\frac{\mu}{2}}-2}\!\! A   \epsilon_{k k} \det \wi{A}_{k k}\hs \ov z _{k \hs [n+1-k]}\epsilon_{n+1-s\hs r} \det \wi{A}_{[n+1-s]\hs r}\hs z _{[n+1-s]\hs s}}{\det^{\frac{\mu}{2}} \!\!  A  - |w|^2}+\right. \\
& +  \frac{{\frac{\mu}{2}}\det^{{\frac{\mu}{2}}-1}\!\! A  \left( \epsilon_{[n+1-s]\hs r}  \epsilon'_{kk}\det \wi A_{[n+1-s]\hs r,k k}\hs\ov z_{k \hs [n+1-k]}\hs z _{[n+1-s]\hs s} +\epsilon_{k r} \det \wi{A}_{k r}\hs \delta_{[n+1-k]\hs s}\right)}
{\det^{\frac{\mu}{2}} \!\! A  - |w|^2}+\\ 
& \left.- \frac{\left({\frac{\mu}{2}}\det^{{\frac{\mu}{2}}-1}\!\! A  \right)^2  \epsilon_{[n+1-s]\hs r} \det \wi{A}_{[n+1-s]\hs r}\hs z _{[n+1-s]\hs s}  \epsilon_{kk} \det \wi{A}_{kk}\hs \ov z _{k \hs [n+1-k]}}{\left(\det^{\frac{\mu}{2}} \!\! A  - |w|^2\right)^2}\right).
 \end{split}\end{equation}
 
If we assume that $(u,w)\in  f (\Delta_{[n/2]})$, we obtain
 \begin{equation}\label{eqgjkrs}\begin{split} 
& \frac{\de^3\Phi_{\Omega_1,\mu}}{\de z_{j\hs[ n+1-j]}\hs \de z_{k\hs [n+1-k]} \hs \de \ov z_{rs}}=
\frac{{\frac{\mu}{2}}({\frac{\mu}{2}}-1)\det^{{\frac{\mu}{2}}-2}\!\! A   \epsilon_{k k} \det \wi{A}_{k k}\hs \ov z _{k \hs [n+1-k]}\epsilon_{[n+1-s]\hs r} \frac{\de \det \wi{A}_{[n+1-s]\hs r}}{\de z_{j\hs[ n+1-j]}}\hs z_{[n+1-s]\hs s} }{\det^{\frac{\mu}{2}} \!\!  A  - |w|^2}+ \\
& +  \frac{{\frac{\mu}{2}}\det^{{\frac{\mu}{2}}-1}\!\! A  \left( \epsilon_{[n+1-s]\hs r}  \epsilon'_{kk}\frac{\de \det \wi A_{[n+1-s]\hs r,k k}}{\de z_{j\hs [n+1-j]}}\hs\ov z_{k \hs[ n+1-k]}\hs z_{[n+1-s]\hs s} +\epsilon_{k r} \det \wi{A}_{k r}\hs \delta_{[n+1-k]\hs s}\right)}
{\det^{\frac{\mu}{2}} \!\! A  - |w|^2}+\\ 
& - \frac{\left({\frac{\mu}{2}}\det^{{\frac{\mu}{2}}-1}\!\! A  \right)^2  \epsilon_{[n+1-s]\hs r}\frac{\de \det \wi{A}_{[n+1-s]\hs r}}{\de z_{j\hs [n+1-j]}} \hs z_{[n+1-s]\hs s}  \epsilon_{kk} \det \wi{A}_{kk}\hs \ov z _{k \hs [n+1-k]}}{\left(\det^{\frac{\mu}{2}} \!\! A  - |w|^2\right)^2},
 \end{split}\end{equation}
 hence if we also assume
 $s\neq n+1 - r$ and $1\leq j,k \leq [n/2]$ we see that $\frac{\de^3\Phi_{\Omega_1,\mu}}{\de z_{j\hs[ n+1-j]}\hs \de z_{k\hs [n+1-k]} \hs \de \ov z_{rs}}=0$. Thus \eqref{eqgam4} holds true, concluding the proof.
 \end{proof}

\subsection{Cartan--Hartogs domain of the third type}
By \eqref{gench}, \eqref{genphi} and \eqref{genericnorm3}, the Cartan--Hartogs associated to a third type domain is given by:
$$
 M_{\Omega_3[m]}(\mu)=\{(Z,w)\in \Omega_2[n]\times \mathds C \mid  |w|^2< {\det}^\mu\!\left(I_{m}-Z Z^{*}\right)\},
$$
and its Kobayashi metric is described by the K\"ahler potential:
$$
\Phi_{\Omega_3,\mu} (Z,w)= -\log\left( {\det}^\mu\!\left(I_{m}-Z Z^{*}\right)-|w|^2\right).
$$

 \begin{lemma}\label{lempolytot3}
 Let $\varphi\!:\Delta^m\rightarrow \Omega_3[m]$ be the map in \eqref{phizeta3}. Then $f: M_{\Delta^{m}}(\mu) \rightarrow M_{\Omega_3[m]}(\mu)$, $f(z,w)=(\varphi(z),w)$, is a totally geodesic  K\"ahler immersion.
\end{lemma}
\proof 
The proof is similar to those of Lemma \ref{lempolytot} and Lemma \ref{lempolytot2} and therefore is omitted.
\endproof

\subsection{Cartan--Hartogs domain of the fourth type}
By \eqref{gench}, \eqref{genphi} and \eqref{genericnorm4}, the Cartan--Hartogs associated to a fourth type domain is given by:
$$
M_{\Omega_4[n]}(\mu)=\left\{(u,w)\in \Omega_4[n]\times \mathds C \midd  |w|^2<
\left(1+\left|\sum_{j=1}^{n} z_{j}^{2}\right|^{2}-2 \sum_{j=1}^{n}\left|z_{j}\right|^{2}\right)^\mu
\right\},
$$
and a K\"ahler potential for the \Ko\ metric is:
$$
\Phi_{\Omega_4,\mu} (z,w)= -\log\left(\left(1+\left|\sum_{j=1}^{n} z_{j}^{2}\right|^{2}-2 \sum_{j=1}^{n}\left|z_{j}\right|^{2}\right)^\mu
- |w|^2\right).
$$ 
%Observe that the rank of $\Omega_4[n]$ is $r=2$, in particular does not depend on $n$.
 \begin{lemma}\label{lempolytot4}
 Let $\varphi\!:\Delta^2\rightarrow \Omega_4[n]$ be the map in \eqref{phizeta4}. Then $f\!: M_{\Delta^2}(\mu) \rightarrow M_{\Omega_4[n]}(\mu)$, $f(z_1,z_2,w) = \left(\varphi(z_1,z_2) , w\right)$,
is a totally geodesic K\"ahler immersion.
\end{lemma}
\proof 

From Sec. \ref{fourthdomain} the map $\varphi$ is a K\"ahler immersion, thus by Lemma \ref{ki} $f$ also is. It remains to prove that $f$ is totally geodesic, which is equivalent to $\Gamma_{j\hs k}^{\ell}= 0 $ for $\ell > 2$ and $0\leq j,k \leq 2$, where:
$$
\nabla _{\de_{z_{j}}}{\de_{z_{k}}} = \sum_{\ell=1}^n \Gamma_{j\hs k}^{\ell}\, \de_{z_{\ell}}+\Gamma_{j \hs k}^{0}\, \de_{w},\qquad
 \nabla _{\de_w}{\de_{z_{ k }}} = \sum_{\ell=1}^n \Gamma_{0\hs k}^{\ell}\, \de_{z_{\ell}}+\Gamma_{0\hs k}^{0}\, \de_{w}.
$$
%and the expression of $f$ and \eqref{genericnorm4}, 
 We have
\begin{equation*}\begin{split} 
\frac{\de\Phi_{\Omega_4,\mu}}{\de \ov z_{k}} 
=& \frac{-\de \log\left(\left(1+\left|\sum_{\ell=1}^{n} z_{\ell}^{2}\right|^{2}-2 \sum_{\ell=1}^{n}\left|z_{\ell}\right|^{2}\right)^\mu
- |w|^2\right)}
{\de \ov z_k}\\
=&\frac{-\mu\left(1+\left|\sum_{\ell=1}^{n} z_{\ell}^{2}\right|^{2}-2 \sum_{\ell=1}^{n}\left|z_{\ell}\right|^{2}\right)^{\mu-1} \left(2 \ov z_k \sum_{\ell=1}^{n} z_{\ell}^{2} -2 z_k\right)}{\left(1+\left|\sum_{\ell=1}^{n} z_{\ell}^{2}\right|^{2}-2 \sum_{\ell=1}^{n}\left|z_{\ell}\right|^{2}\right)^\mu
- |w|^2},
 \end{split}\end{equation*}
 \begin{equation*}\begin{split} 
\frac{\de^2\Phi_{\Omega_4,\mu}}{\de w \de  \ov z_{k}} 
=\frac{-\mu\left(1+\left|\sum_{\ell=1}^{n} z_{\ell}^{2}\right|^{2}-2 \sum_{\ell=1}^{n}\left|z_{\ell}\right|^{2}\right)^{\mu-1} \left(2 \ov z_k \sum_{\ell=1}^{n} z_{\ell}^{2} -2 z_k\right)\ov w}{\left(\left(1+\left|\sum_{\ell=1}^{n} z_{\ell}^{2}\right|^{2}-2 \sum_{\ell=1}^{n}\left|z_{\ell}\right|^{2}\right)^\mu
- |w|^2\right)^2},
 \end{split}\end{equation*}
 and
 \begin{equation*}\begin{split} 
\frac{\de^2\Phi_{\Omega_4,\mu}}{\de z_{h}\de \ov z_{k}} 
=& \frac{-\mu\left(\mu-1\right)\left(1+\left|\sum_{\ell=1}^{n} z_{\ell}^{2}\right|^{2}-2 \sum_{\ell=1}^{n}\left|z_{\ell}\right|^{2}\right)^{\mu-2} \left(2 \ov z_k \sum_{\ell=1}^{n} z_{\ell}^{2} -2 z_k\right)\left(2  z_h \sum_{\ell=1}^{n} \ov z_{\ell}^{2} -2 \ov z_h\right)}{\left(1+\left|\sum_{\ell=1}^{n} z_{\ell}^{2}\right|^{2}-2 \sum_{\ell=1}^{n}\left|z_{\ell}\right|^{2}\right)^\mu
- |w|^2}\\
+&\frac{-\mu\left(1+\left|\sum_{\ell=1}^{n} z_{\ell}^{2}\right|^{2}-2 \sum_{\ell=1}^{n}\left|z_{\ell}\right|^{2}\right)^{\mu-1} \left(4 \ov z_k  z_{h} -2 \delta_{hk} \right)}{\left(1+\left|\sum_{\ell=1}^{n} z_{\ell}^{2}\right|^{2}-2 \sum_{\ell=1}^{n}\left|z_{\ell}\right|^{2}\right)^\mu
- |w|^2}\\
+&\frac{\mu^2\left(1+\left|\sum_{\ell=1}^{n} z_{\ell}^{2}\right|^{2}-2 \sum_{\ell=1}^{n}\left|z_{\ell}\right|^{2}\right)^{2\mu-2}\left(2 \ov z_k \sum_{\ell=1}^{n} z_{\ell}^{2} -2 z_k\right) \left(2  z_h \sum_{\ell=1}^{n} \ov z_{\ell}^{2} -2 \ov z_h\right)}{\left(\left(1+\left|\sum_{\ell=1}^{n} z_{\ell}^{2}\right|^{2}-2 \sum_{\ell=1}^{n}\left|z_{\ell}\right|^{2}\right)^\mu
- |w|^2\right)^2}.
 \end{split}\end{equation*}
 Hence for
% \begin{equation}\label{eqcond4}
 $(z,w)\in  f (M_{\Delta_2}), \ 2 \geq j,h \geq 1 \text{ and } k > 2$,
% \end{equation}
 we have
$$
\Gamma_{0 \hs h}^{k}= \sum_\ell g^{k,\ov \ell}\hs \frac{\de g_{h \hs\ov\ell}}{\de w}= g^{k,\ov k}\hs \frac{\de g_{h \hs\ov k}}{\de w}
$$
and
$$
\Gamma_{j \hs h}^{k}= \sum_\ell g^{k,\ov \ell}\hs \frac{\de g_{h \hs\ov\ell}}{\de z_j}= g^{k,\ov k}\hs \frac{\de g_{k \hs\ov h}}{\de z_j},
$$
where we used that for $(z,w) \in f (M_{\Delta_2}) = \left\{(z,w) \in M_{\Omega_4[n]}\mid z_3=\dots=z_n=0\right\}$, $k>2$ and $k \neq \ell$ we have $g^{k,\ov \ell}(z,w)= 0$. It is straightforward to check that under this conditions $\frac{\de^3\Phi_{\Omega_4,\mu}}{\de w\de z_{h}\de \ov z_{k}}(z,w) =0$ and  $\frac{\de^3\Phi_{\Omega_4,\mu}}{\de z_j\de z_{h}\de \ov z_{k}}(z,w) =0$, namely that
$\Gamma_{j\hs k}^{\ell}(z,w)= 0$. Therefore $ f (M_{\Delta_2})$ is totally geodesic in $M_{\Omega_4[n]}$. The proof is complete.
\endproof

\subsection{The proof of the Hartogs--Polydisk Theorem} 
We need two further preliminary results.
\begin{lemma}\label{lemautlift}
Let  $\Omega$ be a bounded symmetric domain.
\begin{enumerate}\item If $\phi\!: \Omega\f \Omega$ is an isometric automorphism of $\Omega$ then $\phi$ lifts to an isometric automorphism $\tilde \phi :M_\Omega(\mu) \f M_\Omega(\mu)$ defined by 
$$
\tilde \phi (z,w) = \left(\phi(z), e^{\mu h_\phi(z)}{w}\right),
$$
for an appropriate holomorphic function $h_\phi:\Omega\f \C$.
\item If $\phi\!: \Omega\f \Omega$ is an automorphism of $\Omega$ which fix the origin, then $\phi$ lifts to an isometric automorphism $\tilde \phi :M_\Omega(\mu) \f M_\Omega(\mu)$ defined by 
$$
\tilde \phi (z,w) = \left(\phi(z),w\right).
$$
\end{enumerate}
\end{lemma}
\begin{proof}
Let $\phi:\Omega\f\Omega$ be an isometric automorphism of $\Omega$. Then $\phi$ satisfies: 
$$
\de\deb \log \left(N \left(\phi(z),\ov {\phi( z)}\right)\right) = \de\deb \log \left(N \left(z,\ov z\right)\right),
$$ 
and hence $N \left(\phi(z),\ov {\phi( z)}\right) = N \left(z,\ov z\right) e^{h_\phi(z)+\ov h_\phi(z)}$ for an opportune holomorphic function $h_\phi:\Omega\f \C$. The holomorphic map $f:M_{\Omega} (\mu)\f M_\Omega (\mu)$ defined by:
$$
f(z,w)=\left(\phi(z),e^{\mu h_\phi(z)}w\right),
$$
is well defined, as $\left|e^{\mu h_\phi(z)}w\right|^2<\left|e^{\mu h_\phi(z)}\right|^2 N \left(z,\ov z\right) = N \left(\phi(z),\ov {\phi( z)}\right)$, and it is an isometry of $M_\Omega(\mu)$, since:
$$
\de\deb \log \left(N^\mu \left(\phi(z),\ov {\phi( z)}\right)-\left|e^{\mu h_\phi(z)}w\right|^2\right) = \de\deb \log \left(N ^\mu\left(z,\ov z\right)-|w|^2\right). 
$$
For the second part, it is enough to recall that automorphisms of $\Omega$ that fix the origin preserves the minimal polynomial $N_\Om$ (see e.g \cite[Prop. III.2.7]{Bert} or \cite[Section 2.2]{MZ}), thus in this case $h_\phi=0$.
\end{proof}

\begin{prop}\label{corpolgto} 
Let $\Delta^r \subset {\Omega}$ be an  $r$-dimensional totally geodesic polydisk of a bounded symmetric domain of classical type of rank $r$. Then
\begin{equation*}%\label{eqincl}
C_{\Delta^r}=\left\{(z,w)\in M_{\Omega}(\mu) \mid z\in \Delta ^r\right\}
\end{equation*}
is  a totally geodesic  K\"ahler submanifold of $M_{\Omega}(\mu)$ biholomorphically isometric to $M_{\Delta^r}(\mu)$.
\end{prop}
\proof By $(1)$ of Lemma \ref{lemautlift} we can assume without loss of generality that $\Delta^r$ passes through the origin. Observe that $N_{\Delta^r}=N_{\W{\mid_{\Delta^r} }}$ (see \cite[Proposition VI.2.4  and  VI.3.6]{roos}). Now the proof is an immediate consequence of lemmata \ref{lempolytot}, \ref{lempolytot2}, \ref{lempolytot3}, \ref{lempolytot4}, \ref{lemautlift} and the Polydisk Theorem that assure us that $\Aut_0 (\Omega)$ acts transitively on the set of the $r$-dimensional totally geodesic polydisk through the origin of $\W$ (see also \cite[Theorem VI.3.5]{roos}).
\endproof

\begin{proof}[Proof of Theorem \ref{thmpoly}]
Let $X \in T_{(z,w)}M_ \Omega(\mu)$ be a fixed tangent vector.   
Consider the decomposition $X=X_1+X_2$, where $X_1\in T_z\Om$ and $X_2\in \C$. From the Polydisk Theorem we know that there exists a totally geodesic polydisk $\Delta^r\subset \Om$, through $z$, such that $X_1\in  T_z \Delta^r$.  By Proposition \ref{corpolgto} we know that 
$
\left\{(z,w)\in M_{\Omega}(\mu) \mid z\in \Delta ^r\right\}
$
is the Cartan-Hartogs $M_{\Delta^r}(\mu)$ realized as a totally geodesic  K\"ahler submanifold of $M_{\Omega}(\mu)$. The proof is complete by observing that by construction $X \in T_0 \Delta^r \times \C \cong  T_{(z,w)} M_{\Delta^r}(\mu)$. 
\end{proof}

\section{Proof of Theorem \ref{thmtotg}}
In order to proof Theorem \ref{thmtotg} we need the following lemma, which generalize Proposition \ref{corpolgto} to polydisks of dimension less than the rank of $\Omega$.
\begin{lemma}\label{lemtotghp}
Let $\Delta ^r \subset \Delta^n$ be a totally geodesic $r$-dimensional polydisk of an $n$-dimensional polydisk. Then
\begin{equation}\label{eqincl}
%C_{\Delta^r}=
\left\{(z,w)\in M_{\Delta^n}(\mu) \mid z\in \Delta ^r\right\}
\end{equation}
is  a totally geodesic  K\"ahler submanifold of $M_{\Delta^n}(\mu)$ biholomorphic isometric to $M_{\Delta^r}(\mu)$.
\end{lemma}
\proof
Let us first show that the inclusion $i_j:\C {\rm H}^1\f \Delta^n$ of $\C {\rm H}^1$ in the  $j$-th factor of  $\Delta^r$, 
is a holomorphic and totally geodesic immersion of $\C {\rm H}^1$ in $ \Delta^n$. Let us denote by  $ K^{\C {\rm H}^1}$ and $K^{\Delta^n}$ the holomorphic sectional curvatures  of ${\C {\rm H}^1}$ and ${\Delta^n}$ respectively. We have (see \cite[Propostion IX.9.2]{KN2}),
$$
K^{\C {\rm H}^1} (X) = K^{\Delta^n} ({i_j}_* X ) 
  =  \sum_{\ell = 1}^n  K^{\Delta^n}\!\!\left(a_\ell \frac{\de}{\de z_\ell} \right),\qquad \forall X\in T_z\C {\rm H}^1,
$$
where ${i_j}_{*}(X)=\sum_{\ell = 1}^n a_\ell \hs \frac{\de}{\de z_\ell} $.  We conclude that all but one of the $a_1,\dots,a_n$  are forced to be zero.  
We can therefore assume, without loss of generality, that $\C{\rm H}^1\times\cdots\times\C{\rm H}^1=\Delta^r= \left\{z \in \Delta^n \mid z_{j}=0,\ j>r\right\}$.  
Clearly $\left(z_1,\dots,z_r,w\right)\xmapsto{f} \left(z_1,\dots,z_r,0,\dots,0,w\right)$ defines an holomorphic isometric immersion of $M_{\Delta^r}(\mu)$ in $M_{\Delta^n}(\mu)$, in order to complete the proof of the lemma we are going to prove that it is also totally geodesic. 

 Let $\nabla$ be the Levi-Civita connection  of $M_{\Delta^n} (\mu)$, let us denote $\frac{\de}{\de z_0}=\frac{\de}{\de w}$ and let $\Gamma_{ij}^k$ be the associated Christoffel symbols defined by $\nabla_{\frac{\de}{\de z_i}}\frac{\de}{\de z_j}=\sum_{k=0}^n \Gamma_{ij}^k \frac{\de}{\de z_k}$. In order to prove that $f$ is totally geodesic we need to show that $\Gamma_{ij}^k=0$ for $0\leq i,j \leq r$ and  $k > r$. For $k , \ell > 0$ and  $k  \neq \ell$, we have
\begin{equation*}\begin{split} 
g_{k \ov \ell} &=- \frac{i}{2} \de_{z_k} \de_{\ov z_\ell}\log\left(  \prod_{j=1}^n(1-|z_j|^2)^\mu-|w|^2\right)\\
& =   \frac{i}{2} \de_{z_k} \frac{ \mu\hs z_\ell \hs (1-|z_\ell|^2)^{\mu-1} \prod_{j=1, j\neq \ell}^n(1-|z_j|^2)^\mu}{\prod_{j=1}^n(1-|z_j|^2)^\mu-|w|^2}\\
& =   \frac{i}{2} \frac{\mu^2\hs z_\ell \hs \ov z_k \hs (1-|z_\ell|^2)^{\mu-1}(1-|z_k|^2)^{\mu-1} \prod_{j=1, j\neq \ell}^n(1-|z_j|^2)^\mu\prod_{j=1, j\neq k}^n(1-|z_j|^2)^\mu}{ \left(\prod_{j=1}^n(1-|z_j|^2)^\mu-|w|^2\right)^2}\\
& - \frac{i}{2} \frac{ \mu^2\hs z_\ell \hs \ov z_k \hs (1-|z_\ell|^2)^{\mu-1} (1-|z_k|^2)^{\mu-1} \prod_{j=1, j\neq \ell,k}^n(1-|z_j|^2)^\mu}{\prod_{j=1}^n(1-|z_j|^2)^\mu-|w|^2}
\end{split}\end{equation*}
and
\begin{equation*}\begin{split} 
g_{0 \ov \ell} &=- \frac{i}{2} \de_{w} \de_{\ov z_\ell}\log\left(  \prod_{j=1}^n(1-|z_j|^2)^\mu-|w|^2\right)\\
& =   \frac{i}{2} \de_{w} \frac{ \mu\hs z_\ell \hs (1-|z_\ell|^2)^{\mu-1} \prod_{j=1, j\neq \ell}^n(1-|z_j|^2)^\mu}{\prod_{j=1}^n(1-|z_j|^2)^\mu-|w|^2}\\
& =    \frac{i}{2}  \frac{\ov w \hs \mu\hs z_\ell \hs (1-|z_\ell|^2)^{\mu-1} \prod_{j=1, j\neq \ell}^n(1-|z_j|^2)^\mu}{\prod_{j=1}^n(1-|z_j|^2)^\mu-|w|^2}.
\end{split}\end{equation*}
Therefore, for $k> r$ and $z_k=0$, we get
$$
\Gamma_{ij}^k = \sum _ \ell g_{k \ov \ell} \hs\frac{\de g _{j\ell}}{\de z _i} = g_{k \ov k} \hs\frac{\de g _{j \ov k }}{\de z _i} =0,
 $$
for any $i,j \neq k$. The proof is complete.
\endproof

\begin{proof}[Proof of Theorem \ref{thmtotg}]
As $\W'$ is a totally geodesic \K\ submanifold of the bounded symmetric domain $\W$, it is an HSSNCT and therefore can be realized as a bounded symmetric domain $\W'\subset \C^m$, where $m=\dim (\W')$. With a slight abuse of notation, let us denote by $f: \W'\subset \C^m \f \W'\subset \C^n$ the totally geodesic  K\"ahler immersion  of $\W'$ in $\W$. Without loss of generality (up to automorphisms of ${\Omega}$ and $\W '$) we can assume $f(0)=0$. Once observed that $N_{\W'} = N_{\W_{\mid \W '}}$  it is easy to verify that $\tilde f:M_{\W'}(\mu)\f M_{\W}(\mu)$ given by 
$
\tilde f (z,w) = (f(z),w)  
$
defines a K\"ahler embedding, with $C_{\W'}=f\left(M_{\W'}(\mu)\right)\simeq M_{\Omega'}(\mu)$.

It remains to prove that $C_{\W'}$ is totally geodesic in $M_{ \Om }(\mu)$. Let $p\in C_{\W'} \subset M_{ \Om }(\mu)$ and let $X\in T_p C_{\W'} \subset T_{p}M_{\Omega}(\mu)$. We want to prove that  the geodesic $\gamma$ of $M_{ \Om }(\mu)$ with $\gamma(0)=p$ and $\gamma'(0)=X$ is also a geodesic of $C_{\W'}$.   By Theorem \ref{thmpoly} and Proposition \ref{corpolgto}, we know that there exist t.g. \K\ immersed polydisks $\Delta^{r'} \subset \Om '$ and $\Delta^{r} \subset \Om $ such that the associated Hartogs-Polydisk $C_{\Delta^{r'}}\subset M_{ \Om' }(\mu)$ and $C_{\Delta^{r}}\subset M_{ \Om}(\mu)$ are totally geodesics (here $r'$ and $r$ are the ranks of $\W$ and $\W$ respectively).
Using similar argument to that used in first part of the proof of Lemma \ref{lemtotghp} we can see that $\Delta' \cap \Delta$ is a t.g. polydisk of  $\Omega '$ (and therefore of $\Omega$). By Lemma \ref{lemtotghp} we conclude that $C_{\Delta' \cap \Delta}=\left\{(z,w)\in M_{\W}(\mu)\right\}$ is a totally geodesic \K\ submanifold of $C_{\W'}$ and $M_{\W}(\mu)$ at the same time. It is a simple observation that $p\in C_{\Delta' \cap \Delta}$ and $X\in T_pC_{\Delta' \cap \Delta}$, hence $\gamma$ is a geodesic of $M_{\W'}(\mu)$ as wished. 
\end{proof}

\section{Proof of Theorem \ref{autlingeo}}
We start this section giving the explicit expression of a holomorphic and isometric immersion $f$ of $(M_{\Delta^r}(\mu),\omega_{\Delta^r}(\mu))$ in $(l^2(\mathds C),\omega_0)$. 
\begin{lemma}\label{kahler}
The holomorphic map $f\!:M_{\Delta^r}(\mu)\rightarrow l^2(\mathds C)$ given by:
$$
f(z,w)=\left(\psi_1,\dots, \psi_r,\Psi\right),
$$
where for $j=1,\dots, r$:
\begin{equation}\label{psij}
\psi_j:=\sqrt\mu\left(z_j,\dots, \frac{z_j^k}{\sqrt k},\dots\right),
\end{equation}
\begin{equation}\label{Psi}
\Psi:=\left(\dots, \frac1{\sqrt a}\sqrt{{\mu a +k_1-1\choose k_1}\cdots{\mu a +k_r-1\choose k_1}}z_1^{k_1}\dots z_r^{k_r}w^a,\dots\right),
\end{equation}
for $k=(k_1,\dots, k_r)$, $|k|=0,1,2,\dots$, and $a=1,2,\dots$, satisfies $f^*\omega_0=\omega_{\Delta^r}(\mu)$.
\end{lemma}
\begin{proof}
We have:
$$
\sum_{j=0}^\infty |f_j|^2=\mu\sum_{j}\sum_{k_j}\frac{|z_j|^{2k_j}}{k_j}+\sum_{k,a}{{\mu a +k_1-1\choose k_1}\cdots{\mu a +k_r-1\choose k_r}}|z_1|^{2k_1}\dots |z_r|^{2k_r}\frac{|w|^{2a}}{a},
$$
(to avoid confusion the sums are always taken in the parameters' range) and:
$$
\sum_{j}\sum_{k_j}\frac{|z_j|^{2k_j}}{k_j}=-\sum_{j=1}^r\log(1-|z_j|^2)=-\log\left(\prod_{j=1}^r(1-|z_j|^2)\right),
$$
\begin{equation}
\begin{split}
\sum_{k,a}{{\mu a +k_1-1\choose k_1}\cdots{\mu a +k_r-1\choose k_r}}|z_1|^{2k_1}\dots |z_r|^{2k_r}\frac{|w|^{2a}}{a}=&\sum_{a=1}^\infty\frac{|w|^{2a}}{a\prod_{j=1}^r(1-|z_j|^2)^{\mu a}}\\
=&-\log\left(1-\frac{|w|^{2}}{\prod_{j=1}^r(1-|z_j|^2)^{\mu}}\right),
\end{split}\nonumber
\end{equation}
which imply:
$$
\sum_{j=0}^\infty |f_j|^2=-\log\left(\prod_{j=1}^r(1-|z_j|^2)^{\mu}-|w|^{2}\right),
$$
as requested.
\end{proof}

We use Lemma \ref{kahler} to obtain geodesic equations for $M_{\Delta^r}(\mu)$. From \eqref{psij}, deriving twice $\psi_j(\gamma(t))$ w.r.t. $t$ gives:
$$
\psi_j(\gamma)''=\sqrt\mu\left(\ddot u_j,\dots, \frac{k(u_j^{k-1}\dot u_j)'}{\sqrt k},\dots\right),
$$
and, denoting by $A(\mu,a,k):=\frac1{\sqrt a}\sqrt{{\mu a +k_1-1\choose k_1}\cdots{\mu a +k_r-1\choose k_1}}$, from \eqref{Psi} we get:
%$$
%\Psi(\gamma):=\left(\dots, A(\mu,a,k)u_1^{k_1}\dots u_r^{k_r}v^a,\dots\right).
%$$
$$
\Psi(\gamma)'':=\left(\dots, A(\mu,a,k)(u_1^{k_1}\dots u_r^{k_r}u_w^a)'',\dots\right).
$$

The tangent space $T_{f(\gamma)}f(M_{\Delta^r}(\mu))$ is spanned by
$$
\nabla f(\gamma)=\left(\partial_1 f,\dots, \partial_r f,\partial_wf\right)(\gamma),
$$
and the condition for $\gamma$ to be a geodesic is equivalent to the system:
\begin{equation}\label{geodesic}
\langle f(\gamma)'',\partial_1f\rangle=\dots =\langle f(\gamma)'',\partial_rf\rangle=\langle f(\gamma)'',\partial_wf\rangle=0,
\end{equation}
namely:
\begin{equation}\label{gammaw}
\langle f(\gamma)'',\partial_wf\rangle=\sum_{k,a}aA^2(\mu,a,k)(\bar u_1^{k_1}\dots\bar u_r^{k_r}\bar u_w^a)'' u_1^{k_1}\dots  u_r^{k_r} u_w^{a-1}=0,
\end{equation}
and for $s=1,\dots,r$:
\begin{equation}\label{gammas}
\begin{split}
\langle f(\gamma)'',\partial_sf\rangle=&\mu\sum_{k=1}^\infty u^{k-1}_s(\bar u_s^{k})''+\\
&+\sum_{k,a}k_sA^2(\mu,a,k)(\bar u_1^{k_1}\dots \bar u_r^{k_r}\bar u_w^a)''u_1^{k_1}\dots u_s^{k_s-1}\dots u_r^{k_r}v^{a}=0.
\end{split}
\end{equation}

Let us now prove Theorem \ref{autlingeo}.
\begin{proof}[Proof of Theorem \ref{autlingeo}]
Let $\gamma$ be a geodesic with linear support in $M_\Omega(\mu)$, passing through $(\zeta,0)$ with direction $\xi$. By Lemma \ref{lemautlift} up to automorphisms we can assume $\zeta=0$ and by the Hartogs polydisk Theorem $\gamma$ is contained in an Hartogs polydisk $M_{\Delta^r}(\mu)$ passing through $0$ with direction $\xi$.Then $\gamma$ is a geodesic with linear support passing through the origin in $M_{\Delta^r}(\mu)$ and conclusion follows by Lemma \ref{xizero} below.
\end{proof}

\begin{lemma}\label{xizero}
If $\gamma(t)=(\xi_1v(t),\dots, \xi_rv(t),\xi_0v(t))$ is a geodesic in $M_{\Delta^r}(\mu)$, then either $\gamma\subset \Delta^r=M_{\Delta^r}(\mu)\cap \{w=0\}$ or $\gamma\subset \mathds C{\rm H}^1=M_{\Delta^r}(\mu)\cap \{z=0\}$ or $r=1=\mu$, i.e. $M_{\Delta^r}(\mu)\simeq\mathds C{\rm H}^2$. 
\end{lemma}
\begin{proof}
A geodesic in $M_{\Delta^r}(\mu)$ must satisfy \eqref{gammaw} and \eqref{gammas}. Plugging $\gamma$ respectively into \eqref{gammaw} and \eqref{gammas} gives:
\begin{equation}\label{gammaw1}
\bar \xi_0\sum_{k,a}aA^2(\mu,a,k)|\xi_1|^{2k_1}\cdots|\xi_r|^{2k_r}|\xi_0|^{2a-2}(\bar v(t)^{|k|+a})'' v(t)^{|k|+a-1} =0,
\end{equation}
\begin{equation}\label{gammas1}
\begin{split}
&\mu\bar \xi_s\sum_{k=1}^\infty |\xi_s|^{2(k-1)}v(t)^{k-1}(\bar v(t)^{k})''+\\
&+\bar \xi_s\sum_{k,a}k_sA^2(\mu,a,k)|\xi_1|^{2k_1}\cdots|\xi_s|^{2(k_s-1)}\cdots|\xi_r|^{2k_r}|\xi_0|^{2a}(\bar v(t)^{|k|+a})''v(t)^{|k|+a-1}=0.
\end{split}
\end{equation}
for $s=1,\dots,r$, $|k|=0,1,2,\dots$, and $a=1,2,\dots$. Evaluating at $t=0$ we get:
$$
\xi_0{ \ddot v(0)} =\mu\xi_1{ \ddot v(0)}=\dots=\mu\xi_r{ \ddot v(0)}=0,
$$
and since $\xi_0$, $\xi_j$, $j=1,\dots, r$ cannot be all vanishing, it implies $\ddot v(0)=0$.

Taking into account that $v(0)=\ddot v(0)=0$ and $\dot v(0)=1$, deriving \eqref{gammaw1} and \eqref{gammas1} once with respect to $t$ and evaluating at $t=0$ gives:
\begin{equation}\label{derprima}
\bar \xi_0\left( v'''(0)+2|\xi_0|^{2}+2\mu\sum_{j=1}^r|\xi_j|^2\right)=0=\mu\bar\xi_s \left(v'''(0)+2 |\xi_s|^2+2|\xi_0|^2\right).
\end{equation}
If $\xi_0=0$ or $\xi_s=0$ for all $s=1,\dots, r$ then, since by Theorem \ref{thmtotg} $M_{\Delta^r}(\mu)\cap \{z=0\}$ and $M_{\Delta^r}(\mu)\cap \{w=0\}$ are totally geodesic in $M_{\Delta^r}(\mu)$, $\gamma\subset \Delta^r=M_{\Delta^r}(\mu)\cap \{w=0\}$ or $\gamma\subset \mathds C{\rm H}^1=M_{\Delta^r}(\mu)\cap \{z=0\}$. Thus, assume that $\xi_0\neq 0$ and at least one between the $\xi_j$'s is different from $0$. From \eqref{derprima} we get:
\begin{equation}\label{muxi}
\mu\sum_{ j=1}^r|\xi_j|^2=|\xi_s|^2,  \quad {\textrm{for any}}\ s=1,\dots, r,
\end{equation}
which implies that all the $\xi_s$'s are equal in module, and thus $r\mu =1$.

To conclude that $r=\mu=1$, we need to consider the third order derivative of \eqref{gammaw1} and \eqref{gammas1} evaluated at $t=0$. Observe first that:
$$
 \left[(v(t)^{2})''v(t)\right]'''(0)=26 v'''(0),\quad  \left[(v(t)^{3})''v(t)^{2}\right]'''(0)=36,
$$
and recall that from \eqref{derprima} we get $v'''(0)=-2(|\xi_0|^{2}+|\xi_s|^{2})$.
Deriving three times \eqref{gammaw1} with respect to $t$ and evaluating at $t=0$ we get:
\begin{equation}
\begin{split}
\bar\xi_0\left[v^{(v)}(0)+\left(|\xi_0|^{2}+\mu\sum_{j=1}^r |\xi_j|^{2}\right)\right.&%36\ddot v(0)^2\dot v(0)+
26 v'''(0)+36\left(|\xi_0|^{4}+2\mu |\xi_0|^{2}\sum_{j=1}^r|\xi_j|^2+\right.\\
&\left.\left.+\mu(\mu-1)\sum_{j=1}^r|\xi_j|^4+\mu^2\sum_{j, k=1}^r|\xi_j|^2|\xi_k|^2\right)\right]=0,
\end{split}\nonumber
\end{equation}
which by \eqref{muxi} reads:
\begin{equation}\label{finalgeo1}
\bar\xi_0\left[v^{(v)}(0)-16\left(|\xi_0|^{2}+|\xi_s|^{2}\right)^2+36(\mu-1)|\xi_s|^4\right]=0.
\end{equation}
On the other hand, \eqref{gammas1} gives:
$$
\mu\bar \xi_s\left[v^{(v)}(0)+26\left(|\xi_s|^2+|\xi_0|^2\right)v'''(0)+36\left(|\xi_s|^2+|\xi_0|^2\right)^2+36|\xi_0|^2\left(\mu\sum_{j=1}^r|\xi_j|^2-|\xi_s|^2\right)\right]=0,
$$
i.e.:
\begin{equation}\label{finalgeo2}
\mu\bar \xi_s\left[v^{(v)}(0)-16\left(|\xi_s|^2+|\xi_0|^2\right)^2\right]=0.
\end{equation}
Comparing \eqref{finalgeo1} and \eqref{finalgeo2} we get $\mu=r=1$ and we are done.
\end{proof}

\end{document}